\documentclass[a4paper,reqno]{amsart}
\title{Notes on transformations in integrable geometry}
\author{Fran Burstall}
\address{Department of Mathematical Sciences\\ University of Bath\\
  Bath BA2 7AY\\UK} \email{feb@maths.bath.ac.uk}
\usepackage{amssymb,amsmath,fullpage,amsthm,graphicx,paralist,para}
\usepackage{tikz}
\usepackage[pdftex,plainpages=false,pdfpagelabels,linktocpage,pdfstartview=FitH,colorlinks=true,linkcolor=blue,citecolor=magenta]{hyperref}
\usepackage[initials,msc-links]{amsrefs}[2007/10/22]
\newtheorem{thm}{Theorem}[section]
\newtheorem{prop}[thm]{Proposition}

\newtheorem{lem}[thm]{Lemma}
\newtheorem*{proposition*}{Proposition}
\newtheorem*{theorem*}{Theorem}

\theoremstyle{definition}
\newtheorem*{xmpls}{Examples}

\newtheorem*{defn}{Definition}

\newtheorem*{ex}{Exercise}

\theoremstyle{remark}
\newtheorem*{rem}{Remark}
\newtheorem*{rems}{Remarks}
\numberwithin{equation}{section}

\newcommand{\R}{\mathbb{R}}
\newcommand{\C}{\mathbb{C}}
\newcommand{\Z}{\mathbb{Z}}

\renewcommand{\P}{\mathbb{P}}

\renewcommand{\d}{\mathrm{d}}
\newcommand{\del}{\partial}

\newcommand{\<}{\langle}
\renewcommand{\>}{\rangle}

\newcommand{\unC}{\underline{\C}}
\newcommand{\fso}{\mathfrak{so}}

\newcommand{\rO}{\mathrm{O}}
\newcommand{\rSO}{\mathrm{SO}}

\newcommand{\cL}{\mathcal{L}}
\newcommand{\cD}{\mathcal{D}}
\newcommand{\cN}{\mathcal{N}}

\newcommand{\Ad}{\mathrm{Ad}}
\newcommand{\ad}{\mathrm{ad}}

\newcommand{\hd}{\hat{\d}}

\newcommand{\hL}{\hat{L}}

\newcommand{\hf}{\hat{f}}
\newcommand{\hN}{\hat{N}}
\newcommand{\g}{\mathfrak{g}}
\newcommand{\p}{\mathfrak{p}}
\newcommand{\set}[1]{\{#1\}}

\newcommand{\I}{\mathrm{I}}
\newcommand{\II}{\mathrm{I\kern-1ptI}}
\newcommand{\III}{\mathrm{I\kern-1ptI\kern-1ptI}}
\newcommand{\half}{\tfrac12}
\newcommand{\norm}[1]{\Vert#1\Vert}

\DeclareSymbolFont{script}{U}{eus}{m}{n}
\DeclareSymbolFontAlphabet{\mathscr}{script}
\DeclareMathSymbol{\EuWedge}{0}{script}{"5E}
\newcommand{\Wedge}{\EuWedge}

\begin{document}
\maketitle

\section*{Prospectus}
\label{sec:prospectus}

Roughly speaking, a differential-geometric system, be it smooth,
discrete or semi-discrete, is integrable if it has some or all of the
following properties:
\begin{enumerate}[1.]
\item an infinite-dimensional symmetry group.
\item explicit solutions.
\item algebro-geometric solutions via spectral curves and/or theta
  functions.
\end{enumerate}
In these talks, I shall focus on a manifestation of the first item:
transformations whereby new solutions are constructed from old.  The
theory applies in many situations including:
\begin{itemize}
\item surfaces in $\R^3$ with constant mean or Gauss
  curvature \cite{Bac83,Bia79,Lie79} or,
  more generally, linear Weingarten surfaces in $3$-dimensional
  spaces forms.
\item (constrained) Willmore surfaces in $S^n$ \cite{BurQui14}.
\item projective minimal and Lie minimal surfaces in $\P^3$
  and $S^3$ respectively \cite{BurHer02,Cla12}.
\item affine spheres \cite{BobSch99,FerSch99}.
\item harmonic maps of a surface into a pseudo-Riemannian symmetric
  space \cite{TerUhl00,Uhl89}: this includes many of the preceding examples via some form
  of Gauss map construction.
\item isothermic surfaces in $S^n$
  \cite{Bia05,Bia05a,Cal03,Dar99e,Bur06,Sch01} or, more
  generally, isothermic submanifolds in symmetric
  $R$-spaces \cite{BurDonPedPin11}.
\item M\"obius flat submanifolds of $S^n$
  \cite{BurCal10,BurCal}: these include Guichard surfaces
  and conformally flat submanifolds with flat normal
  bundles, in particular, conformally flat hypersurfaces.
\item omega surfaces in Lie sphere geometry \cite{Pem15}.
\item curved flats in pseudo-Riemannian symmetric spaces: these are
  related to the last four items.
\item self-dual Yang--Mills fields \cite{Cra87}: many of our low-dimensional
  examples are dimensional reductions of these \cite{War85}.
\end{itemize}

In these lectures, I shall discuss this theory via two examples:
\emph{$K$-surfaces} (surfaces in $\R^3$ of constant Gauss curvature)
and \emph{isothermic surfaces}.  In both cases, I will emphasise:
\begin{itemize}
\item the geometry of transformations
\item a gauge-theoretic approach of wide applicability.
\end{itemize}

\section{$K$-surfaces}
\label{sec:k-surfaces}

\subsection{Classical surface geometry}
\label{sec:class-surf-geom}

Let $f:\Sigma^2\to\R^3$ be an immersion with Gauss map $N:\Sigma\to
S^2$.  Thus:
\begin{equation*}
  N\cdot \d f=0.
\end{equation*}
These yield three invariant quadratic forms:
\begin{align*}
  \mathrm{I}&:=\d f\cdot\d f\\
  \II&:=-\d f\cdot\d N\\
  \III&:=\d N\cdot \d N
\end{align*}
and the famous theorem of Bonnet says that the first two determine
$f$ up to a rigid motion.

Lowering an index on $\II$ gives the \emph{shape operator} $S:=-(\d
f)^{-1}\circ\d N$, a symmetric (with respect to $\mathrm{I}$)
endomorphism on T$\Sigma$.  The shape operator has eigenvalues
$\kappa_1,\kappa_2$, the \emph{principal curvatures} from which the
\emph{mean curvature} $H$ and the \emph{Gauss curvature} $K$ are
given by
\begin{align*}
  H&:=\half(\kappa_1+\kappa_2)\\
  K&:=\kappa_1\kappa_2.
\end{align*}
Further, the Cayley--Hamilton theorem applied to $S$ gives:
\begin{equation}
  \label{eq:1}
  \III-2H\II+K\mathrm{I}=0.
\end{equation}

\subsection{Lelieuvre's Formula}
\label{sec:lelieuvres-formula}

Let us suppose that $K<0$ and write $K=-1/\rho^2$.  In this case, $f$
admits asymptotic coordinates $\xi,\eta$, thus:
\begin{equation*}
  N_{\xi}\cdot f_{\xi}=0=N_{\eta}\cdot f_{\eta}.
\end{equation*}
It follows at once that there are functions $a,b$ so that
\begin{equation*}
  a(N\times N_{\xi})=f_{\xi}\qquad b(N\times N_{\eta})=f_{\eta}.
\end{equation*}
Now the symmetry of $\II$: $N_{\xi}\cdot f_{\eta}=N_{\eta}\cdot
f_{\xi}$, rapidly yields $a=-b$ while \eqref{eq:1} evaluated on
$\del_{\xi}$ gives $a^2=\rho^2$ so that we have \emph{Lelieuvre's
  Formula}:
\begin{equation}
  \label{eq:2}
  \begin{split}
    \rho(N\times N_{\xi})&=f_{\xi}\\
    \rho(N\times N_{\eta})&=-f_{\eta}.
  \end{split}
\end{equation}
Cross-differentiating \eqref{eq:2} gives us two formulae for the
tangential component of $f_{\xi\eta}$:
\begin{equation*}
  f_{\xi\eta}{}^T=\rho_{\eta}N\times N_{\xi}+\rho N\times
  N_{\xi\eta}=
  \rho_{\xi}N\times N_{\eta}-\rho N\times
  N_{\xi\eta}.
\end{equation*}
From this and the linear independence of $N\times N_{\xi}$ and
$N\times N_{\eta}$, we easily see that $f_{\xi\eta}{}^T=0$ if and
only if $\rho$ is constant if and only if $N\times N_{\xi\eta}=0$,
that is, $N:(\Sigma,\II)\to S^2$ is a harmonic map.  Finally,
$f_{\xi\eta}{}^T=0$ if and only if
\begin{equation*}
  (f_{\xi}\cdot f_{\xi})_{\eta}=0=(f_{\eta}\cdot f_{\eta})_{\xi}.
\end{equation*}
We conclude:
\begin{thm}
  \label{th:1}
  The following are equivalent:
  \begin{compactitem}
    \item $K$ is constant.
    \item $N:(\Sigma,\II)\to S^2$ is harmonic.
    \item Asymptotic coordinates can be chosen so that $\Vert
      f_{\xi}\Vert=1=\Vert f_{\eta}\Vert$. We say such coordinates are
      \emph{Tchebyshev} for $f$.
  \end{compactitem}
\end{thm}
\subsection{Geometry of $K$-surfaces}
\label{sec:geometry-k-surfaces}

\begin{defn}
  $f:\Sigma\to\R^3$ is a \emph{$K$-surface} if it has constant,
  negative Gauss curvature.
\end{defn}

From Theorem~\ref{th:1}, we see that a $K$-surface admits Tchebyshev
coordinates $\xi,\eta$ with respect to which we have
\begin{align*}
  \mathrm{I}&=\d\xi^2+2\cos\omega\d\xi\d\eta+\d\eta^2\\
  \II&=\frac{2}{\rho}\sin\omega\d\xi\d\eta,
\end{align*}
where $\omega$ is the angle between the coordinate directions.

For such $\mathrm{I},\II$, the Codazzi equations are vacuous while
the Gauss equation reads
\begin{equation}
  \label{eq:3}
  \omega_{\xi\eta}=\frac{1}{\rho^2}\sin\omega.
\end{equation}
Thus any solution of \eqref{eq:3} gives rise to a $K$-surface.

Let us now turn to the symmetries of the situation:

\subsubsection{B\"acklund transformations}
\label{sec:backl-transf}

Let $f$ be a $K$-surface and, following Bianchi (1879) and B\"acklund
(1883), seek $\hf:\Sigma\to\R^3$ such that:
\begin{itemize}
\item $\hf-f$ is tangent to both $f$ and $\hf$.
\item $\norm{\hf-f}$ is constant.
\item $\hN\cdot N$ is constant. (Bianchi considered the case
  $\hN\cdot N=0$.)
\end{itemize}
Then:
\begin{enumerate}
\item $\hf$ exists if and only if $f$ is a $K$-surface (and then, by
  symmetry, $\hf$ is a $K$-surface too, in fact with the same value
  of $K$).
\item Given $a>0$, $p_0\in\Sigma$ and a ray $\ell_0\subset
  T_{p_0}\Sigma$, one solves commuting ODE to get (locally) a unique $\hf$ with
  \begin{align*}
    \norm{\hf-f}&=\frac{2\rho}{a+a^{-1}}\\
    \hN\cdot N&=\frac{a^{-1}-a}{a^{-1}+a}\\
    \hf(p_{0})&\in\ell_0.
  \end{align*}
  Thus, if $\hN\cdot N=\cos\theta$, $\norm{\hf-f}=\rho\sin\theta$.

  We write $\hf=f_a$ and say that $f_a$ is a \emph{B\"acklund
    transform} of $f$,
\item $\xi,\eta$ are asymptotic, in fact Tchebyshev, for $\hf$ too.
  In classical terminology, $f$ and $\hf$ are the focal surfaces of a
  \emph{$W$-congruence}.
\item Permutability (Bianchi \cite{Bia92}, 1892): given a
  $K$-surface
  $f$
  and two B\"acklund transforms $f_a, f_b$
  with $a\neq b$,
  one can choose initial conditions so that there is a
  fourth $K$-surface $\hf$ with
  \begin{equation*}
    \hf=(f_a)_b=(f_b)_a.
  \end{equation*}
  \begin{ex}
    $\hf$ is algebraic in $f,f_{a},f_b$.
  \end{ex}
  One can iterate the procedure and so build up a quad-graph of
  $K$-surfaces.  At each point $p\in\Sigma$, the corresponding
  quad-graph of points in $\R^3$ is a discrete $K$-surface in the
  sense of Bobenko--Pinkall.
\end{enumerate}

\subsubsection{Lie transform}
\label{sec:lie-transform}

If $\xi,\eta$ are Tchebyshev coordinates for a $K$-surface $f$, we
have seen that the Gauss--Codazzi equations reduce to the sine-Gordon
equation~\eqref{eq:3} for the angle $\omega$ between coordinate
directions.

However, for $\mu\in\R^{\times}$, we observe that
\begin{equation*}
  \omega^{\mu}(\xi,\eta):=\omega(\mu^{-1}\xi,\mu\eta)
\end{equation*}
also solves \eqref{eq:3} and so gives rise to a new $K$-surface
$f^{\mu}$ with
\begin{equation*}
  \mathrm{I}_{f^{\mu}}=\d\xi^2+2\cos\omega^{\mu}\d\xi\d\eta+\d\eta^2.
\end{equation*}
Such an $f^{\mu}$ is a \emph{Lie transform} of $f$.

\subsection{Harmonic maps and flat connections}
\label{sec:harmonic-maps-flat}

We have seen that $K$-surfaces give rise to harmonic maps
$(\Sigma,\II)\to S^2$.  The converse is also true: let $c$ be a
conformal structure on $\Sigma$ of signature $(1,1)$, $*$ the
Hodge-star of $c$ and $\xi,\eta$ null coordinates.  Let
$N:(\Sigma,c)\to S^2$ so that
\begin{equation*}
  *\d N=N_{\xi}\d\xi-N_{\eta}\d\eta.
\end{equation*}
It is easy to see that $N$ is harmonic if and only if
\begin{equation*}
  \d(N\times *\d N)=0
\end{equation*}
in which case we can locally find $f:\Sigma\to\R^3$ such that
$N\times *\d N=\d f$, that is,
\begin{align*}
  N\times N_{\xi}&=f_{\xi}\\ N\times N_{\eta}&=-f_{\eta}.
\end{align*}
It follows at once that, whenever $f$, equivalently $N$, immerses,
\begin{enumerate}
\item $N\perp f_{\xi},f_{\eta}$ so that $N$ is the Gauss map of $f$.
\item $N_{\xi}\cdot f_{\xi}=0=N_{\eta}\cdot f_{\eta}$ so that
  $\xi,\eta$ are asymptotic for $f$ whence $c=\<\II\>$.
\item $K=-1$ after recourse to \eqref{eq:1}.
\end{enumerate}

\subsubsection{Flat connections}
\label{sec:flat-connections}

The basic observation for all that follows is that harmonic maps give
rise to a holomorphic family of flat connections on the trivial
bundle $\underline{\R}^3:=\Sigma\times\R^3$.  We rehearse this
construction in such a way as to indicate how it generalises to any
(pseudo)-Riemannian symmetric target.

So let $N:(\Sigma,c)\to S^2$ and $\rho^N:\Sigma\to\rO(3)$ be the
reflection across $N^{\perp}$.  The orthogonal decomposition
\begin{equation*}
  \underline{\R}^3=\<N\>\oplus N^{\perp}
\end{equation*}
induces a decomposition of the flat connection $\d$:
\begin{equation*}
  \d=\cD+\cN
\end{equation*}
where $N,\rho^N$ are $\cD$-parallel and $\cN\in\Omega^1(\fso(3))$
anti-commutes with $\rho^N$.

We shall several times have recourse to the identification
$\R^3\cong\fso(3)$ given by
\begin{equation}\label{eq:4}
  v\mapsto(u\mapsto v\times u)
\end{equation}
under which $\cN$ is identified with $N\times\d N$.

The structure equations of the situation express the flatness of $\d$
and read:
\begin{gather*}
  R^{\cD}+\half[\cN\wedge\cN]=0\\
  \d^{\cD}\cN=0,
\end{gather*}
while $N$ is harmonic if and only if $\d *\cN=0$, or, equivalently,
$\d^{\cD}*\cN=0$ since $[*\cN\wedge\cN]$ is always zero.

Now write
\begin{equation*}
  \cN=\cN^++\cN^-
\end{equation*}
where $*\cN^{\pm}=\pm\cN^{\pm}$.  Then we have
\begin{align*}
  \d^{\cD}\cN&=\d^{\cD}\cN^++\d^{\cD}\cN^-  =0\\
  \d^{\cD}*\cN&=\d^{\cD}\cN^+-\d^{\cD}\cN^-
\end{align*}
so that $N$ is harmonic if and only if $\d^{\cD}\cN^{\pm}=0$.

Let $\lambda\in\C^{\times}$ and define a connection $\d_{\lambda}$ on
$\underline\C^3$ by
\begin{equation*}
  \d_{\lambda}=\cD+\lambda\cN^{+}+\lambda^{-1}\cN^{-}.
\end{equation*}
Then, comparing coefficients of $\lambda$ in $R^{d_{\lambda}}$, we have:
\begin{prop}
  $N$ is harmonic if and only if $\d_{\lambda}$ is flat for all $\lambda\in\C^{\times}$.
\end{prop}
We note that $\d_{\lambda}$ has the following four properties:
\begin{enumerate}[(i)]
\item $\lambda\mapsto d^+_{\lambda}$ is holomorphic on $\C$ with a
  simple pole at $\infty$ while $\lambda\mapsto\d^-_{\lambda}$ is
  holomorphic on $\C^{\times}\cup\set\infty$ with a simple pole at $0$.
\item $\rho^N\cdot\d_{\lambda}=\d_{-\lambda}$.  Here, and below, for
  connection $D$ and gauge transformation $g:\Sigma\to\rO(3,\C)$,
  $g\cdot D=g\circ D\circ g^{-1}$, the usual action of gauge
  transformations on connections.
\item $\d_{\bar{\lambda}}=\overline{\d_{\lambda}}$.
\item $\d_1=\d$.
\end{enumerate}
\begin{ex}
  These properties uniquely determine $\d_{\lambda}$.
\end{ex}
Thus:
\begin{prop}
  $N$ is harmonic if and only if there is a family
  $\lambda\mapsto\d_{\lambda}$ of flat connections with properties (i)--(iv).
\end{prop}

\subsection{Spectral deformation}
\label{sec:spectral-deformation}

Let $N$ be harmonic with flat connections $\d_{\lambda}$.  Since
$\d_{\lambda}$ is flat, there is a locally a trivialising gauge
$T_{\lambda}:\Sigma\to\rSO(3,\C)$, that is,
\begin{equation*}
  T_{\lambda}\cdot\d_{\lambda}=\d.
\end{equation*}

Now fix $\mu\in\R^{\times}$ and set
$\d^{\mu}_{\lambda}:=\d_{\lambda\mu}$.  We notice that
$\lambda\mapsto\d^{\mu}_{\lambda}$ has properties (i)--(iii) but
$\d^{\mu}_1=\d_{\mu}$.  It follows then that $\lambda\mapsto
T_{\mu}\cdot\d^{\mu}_{\lambda}$ has (i)--(iv) with respect to
$N^{\mu}:= T_{\mu}N:\Sigma\to S^2$.  We therefore conclude
\begin{thm}
  $N^{\mu}:(\Sigma,c)\to S^2$ is harmonic and so gives rise to a
  $K$-surface $f^{\mu}$.
\end{thm}
Moreover,
\begin{equation*}
  \d N^{\mu}=T_{\mu}(\d_{\mu}N)=T_{\mu}(\mu\d^+N+\mu^{-1}\d^-N).
\end{equation*}
If $\xi,\eta$ are Tchebyshev for $f$, it follows from this that
$f^{\mu}$ has first fundamental form
\begin{equation*}
  I_{f^{\mu}}=\mu^2\d\xi^2+2\cos\omega\d\xi\d\eta+\mu^{-2}\d\eta^2.
\end{equation*}
Thus $\hat{\xi}=\mu\xi$ and $\hat{\eta}=\mu^{-1}\eta$ are Tchebyshev
for $f^{\mu}$ and the corresponding sine-Gordon solution is
$\omega(\mu^{-1}\hat{\xi},\mu\hat{\eta})$.  Otherwise said, $f^{\mu}$
is a Lie transform of $f$.

\subsection{Sym formula}
\label{sec:sym-formula}

Knowing the trivialising gauges $T_{\lambda}$, allows us to compute
Lie transforms without integrations.  Indeed, by definition,
\begin{equation*}
  T_{\lambda}\circ\d_{\lambda}=\d\circ T_{\lambda}
\end{equation*}
and differentiating this with respect to $\lambda$ at $\mu$ yields:
\begin{equation*}
  \Ad_{T_{\mu}}(\del \d_{\lambda}/\del\lambda|_{\mu})=
  \d(\del T_{\lambda}/\del\lambda|_{\mu}T_{\mu}^{-1}).
\end{equation*}
The left side of this reads
\begin{equation*}
  \Ad_{T_{\mu}}(\cN^+-\cN^-/\mu^2)=\frac{1}{\mu}*\cN^{\mu},
\end{equation*}
or, using the identification \eqref{eq:4},
\begin{equation*}
  N^{\mu}\times *\d N^{\mu}= \mu\d(\del T_{\lambda}/\del\lambda|_{\mu}T_{\mu}^{-1}).
\end{equation*}
Thus we have the \emph{Sym formula} \cite{Sym85}:
\begin{equation}
  \label{eq:5}
  f^{\mu}=\mu\del T_{\lambda}/\del\lambda|_{\mu}T_{\mu}^{-1},
\end{equation}
where we use \eqref{eq:4} to view the right side as a map
$\Sigma\to\R^3$.  In particular, taking $\mu=1$ and assuming,
without loss of generality that $T_1=1$, we recover our original
$K$-surface:
\begin{equation}
  \label{eq:8}
  f=\del T_{\lambda}/\del\lambda|_{\lambda=1}.
\end{equation}

\subsection{Parallel sections and B\"acklund transformations}
\label{sec:parall-sect-backl}

Again we start with a harmonic $N:(\Sigma,c)\to S^2$ and its family
of flat connections $\d_{\lambda}$.  We seek to construct a
holomorphic family of gauge transformations
$r(\lambda):\Sigma\to\rSO(3,\C)$ so that the connections
$r(\lambda)\cdot d_{\lambda}$ have properties (i)--(iv) with respect
to a new map $\hN:\Sigma\to S^2$.  Since these gauged connections are
flat, $\hN$ will be harmonic.

We will build our gauge transformations from
$d_{\lambda}$-parallel
subbundles of $\unC^3$
using an avatar of a construction of Terng--Uhlenbeck
\cite{TerUhl00}.  First the algebra: for null line
subbundles $L,L^{*}\leq\unC^3$
with $L\cap L^{*}=\set0$, define
\begin{equation*}
  \Gamma^L_{L^{*}}(\lambda)=
  \begin{cases}
    \lambda&\text{on $L$}\\
    1&\text{on $(L\oplus L^{*})^{\perp}$}\\
    \lambda^{-1}&\text{on $L^{*}$}
  \end{cases}
  :\Sigma\to\rSO(3,\C).
\end{equation*}
The decisive properties of $\Gamma^L_{L^{*}}$ is that it takes values
in semisimple homomorphisms $\C^{\times}\to\rSO(3,\C)$ and that $\Ad
\Gamma^L_{L^{*}}$ has only simple poles at $0$ and $\infty$.

Now fix $a>0$ and choose $L$ so that:
\begin{enumerate}
\item $L$ is $d_{ia}$-parallel.
\item $\rho^NL=\bar{L}$.  Denote this bundle by $L^{*}$.
\item $L\cap L^{*}=\set0$.
\end{enumerate}
\begin{rems}
\item{}
  \begin{enumerate}
  \item The last two conditions amount to demanding that $(L\oplus
    L^{*})^{\perp}$ is a real line tangent to $N$ or, equivalently,
    $f$.  It is on this line that our B\"acklund transform will
    eventually lie.
  \item The conditions are compatible: both $\rho^NL$ and $\bar{L}$
    are $d_{-ia}$-parallel and so coincide as soon as they do so at
    an initial point.
  \item In fact, $L$ is completely determined by the data of a
    unit tangent vector $t$ at a single point $p_0\in\Sigma$. We take
    for $L_{p_{0}}$ and $L^{*}_{p_0}$ the $\pm i$-eigenspaces of
    $v\mapsto t\times v$ and then define $L$ and $L^{*}$ by parallel
    transport.  Of course, condition (3) may fail eventually.
  \end{enumerate}
\end{rems}
Thanks to condition (2), we have
\begin{equation}
\begin{split}
  \rho^N\Gamma^L_{L^{*}}(\lambda)&=\Gamma^L_{L^{*}}(\lambda^{-1})\\
  \overline{\Gamma^L_{L^{*}}(\lambda)}&=\Gamma^{L}_{L^{*}}(1/\bar{\lambda}).
\end{split}\label{eq:6}
\end{equation}

After all this preparation, we finally set:
\begin{equation*}
  r(\lambda)=\Gamma^L_{L^{*}}\left(\bigl(\frac{1+ia}{1-ia}\bigr)
    \bigl(\frac{\lambda-ia}{\lambda+ia}\bigr)\right).
\end{equation*}
We have:
\begin{itemize}
\item $\lambda\mapsto r(\lambda)$ is holomorphic on
  $\P^1\setminus\set{\pm ia}$.
\item $\overline{r(\lambda)}=r(\bar{\lambda})$ so that, in
  particular, $r(\lambda)$ takes values in $\rSO(3)$ for $\lambda\in\R\P^1$.
\item $r(-\lambda)\circ\rho^N\circ r(\lambda)^{-1}$ is independent of
  $\lambda$ and so coincides with $r(\infty)\rho^N
  r(\infty)^{-1}=\rho^{r(\infty)N}$.  Thus
  \begin{equation}
    \label{eq:7}
    r(-\lambda)\circ \rho^N=\rho^{r(\infty)N}\circ r(\lambda).
  \end{equation}
\item $r(1)=1$.
\end{itemize}
We therefore set:
\begin{align*}
  \hN&=r(\infty)N\\\hd_{\lambda}&=r(\lambda)\cdot\d_{\lambda}.
\end{align*}
\begin{prop}
  $\hd_{\lambda}$ has properties (i)--(iv) with respect to $\hN$.
\end{prop}
\begin{proof}
  This is all a straightforward verification except for item (i).
  Since $r(\lambda)$ is holomorphic near zero and infinity,
  $\hd_{\lambda}$ has the same poles there as $\d_{\lambda}$ so the
  main issue is to see that $\lambda\mapsto \hd_{\lambda}$ is
  holomorphic at $\pm ia$.  This is where the fact that $L$ and
  $L^{*}$ are parallel comes in and is an immediate consequence of
  the following
  \begin{lem}
    \label{th:2}
    Let $L,L^{*}$ be null line subbundles,
    $\lambda\mapsto\d_{\lambda}$ \emph{any} holomorphic family of
    connections and $\psi^{\alpha}_{\beta}$ any linear fractional
    transformation with a zero at $\alpha$ and a pole at $\beta$.

    Then
    $\lambda\mapsto\Gamma^L_{L^{*}}(\psi^{\alpha}_{\beta}(\lambda))\cdot\d_{\lambda}$
    is holomorphic at $\alpha$ if and only if $L$ is
    $\d^{\alpha}$-parallel and holomorphic at $\beta$ if and only if
    $L^{*}$ is $\d_{\beta}$-parallel. 
  \end{lem}
\end{proof}
We therefore conclude:
\begin{thm}
  $\hN:(\Sigma,c)\to S^2$ is harmonic with associated flat
  connections $\hd_{\lambda}=r(\lambda)\cdot\d_{\lambda}$.
\end{thm}
It is important that we have control on $\hd_{\lambda}$ since this
allows us to iterate the construction as we shall see below.

Now let us turn to the geometry of the situation.  $N$ and $\hN$ are
the Gauss maps of $K$-surfaces $f$ and $\hf$, both with $K=-1$.  We
compute $\hf$ via the Sym formula: if $T_{\lambda}$ is a trivialising
gauge for $\d_{\lambda}$ with $T_1=1$,
then $T_{\lambda}r(\lambda)^{-1}$ is a trivialising gauge for
$\hd_{\lambda}$ so that \eqref{eq:8} yields
\begin{equation*}
  \hf=f-\del r/\del\lambda|_{\lambda=1}.
\end{equation*}
The chain rule gives
\begin{equation*}
  \del r/\del\lambda|_{\lambda=1}=\frac{2ia}{1+a^2}(\Gamma^L_{L^{*}})'(1)
\end{equation*}
which, under the identification \eqref{eq:4}, is
\begin{equation*}
  \frac{2a}{1+a^2}t
\end{equation*}
for $t$ a real unit length section of $(L\oplus L^{*})^{\perp}\leq
\d f(T\Sigma)$.

We therefore conclude:
\begin{itemize}
\item $\hf-f=-2t/(a+a^{-1})$ and so is tangent to $f$ and of constant
  length.
\item Since $r(\infty)$ is rotation about $t$ through angle $\theta$
  for
  \begin{equation*}
    e^{i\theta}=\frac{1+ia}{1-ia},
  \end{equation*}
  $\hN=r(\infty)N$ is orthogonal to $t$ (so that $\hf-f$ is tangent
  to $\hf$ as well) and
  \begin{equation*}
    \hN\cdot N=\mathrm{Re} \frac{1+ia}{1-ia}=\frac{a^{-1}-a}{a^{-1}+a}.
  \end{equation*}
  Thus $\hf=f_a$, a B\"acklund transform of $f$.
\item The conformal structures of $\II$ and $\hat{\II}$ coincide
  (they are both $c$).  Otherwise said, the asymptotic lines of $f$
  and $\hf$ coincide.
\end{itemize}

\subsection{Permutability}
\label{sec:permutability}

Suppose that we have a $K$-surface and two B\"acklund transforms
$f_a$ and $f_b$ produced from $\d_{ia}$-parallel $L_{a}$ and
$\d_{ib}$-parallel $L_b$, respectively.  We now seek a fourth
$K$-surface
\begin{equation*}
  f_{ab}=(f_{a})_b=(f_b)_a.
\end{equation*}
For this we will need a $\d^a_{ib}$-parallel $\hL_b$ and a
$\d^b_{ia}$-parallel $\hL_a$.  However, since
$\d^a_{ib}=r_a(ib)\cdot\d_{ib}$ etc, we have natural candidates in
\begin{equation}
\begin{split}
  \hL_b&:=r_a(ib)L_b\\\hL_a&:=r_b(ia)L_a.
\end{split}\label{eq:9}
\end{equation}
\begin{ex}
  Check that $\hL_a,\hL_b$ so defined satisfy
  \begin{equation*}
    \rho^a\hL_b=\overline{\hL_b}\qquad
    \rho^b\hL_a=\overline{\hL_a}.
  \end{equation*}
\end{ex}
We therefore have $K$-surfaces $(f_a)_b$ and $(f_b)_a$ with
associated connections $\hat{r}_b(\lambda)\cdot\d^a_{\lambda}$ and
$\hat{r}_a(\lambda)\cdot\d^b_{\lambda}$.  The key to showing these
coincide is the following
\begin{prop}\label{th:4}
  $\hat{r}_br_a=\hat{r}_ar_b$.
\end{prop}
For this we need a lemma which is a discrete version of
Lemma~\ref{th:2}:
\begin{lem}\label{th:3}
  Let $\ell^{\pm},\hat{\ell}^{\pm}$ be two pairs of distinct null
  lines, $\psi^{\alpha}_{\beta}$ a linear fraction transformation
  with a zero at $\alpha$ and a pole at $\beta$ and $\lambda\mapsto
  E(\lambda)$ holomorphic near $\alpha$ and $\beta$.  Then
  \begin{equation*}
    \lambda\mapsto
    \Gamma^{\hat{\ell}^+}_{\hat{\ell}_-}(\psi^{\alpha}_{\beta}(\lambda))E(\lambda)
    (\Gamma^{\ell^+}_{\ell^-}(\psi^{\alpha}_{\beta}(\lambda)))^{-1}
  \end{equation*}
  is holomorphic at $\alpha$ if and only if
  $E(\alpha)\ell^+=\hat{\ell}^+$ and holomorphic at $\beta$ if and
  only if $E(\beta)\ell^-=\hat{\ell}^-$.
\end{lem}
\begin{proof}[Proof of Proposition~\ref{th:4}]
  We show that $R:=\hat{r}_br_{a}r_{b}^{-1}\hat{r}^{-1}_a$ is identically
  $1$.  Note $R$ is holomorphic on $\P^1\setminus\set{\pm ia,\pm ib}$
  and $R(1)=1$.  Now Lemma~\ref{th:3} together with \eqref{eq:9}
  shows that $\hat{r}_br_{a}r_{b}^{-1}$ is holomorphic at $\pm
  ib$ and that $r_{a}r_{b}^{-1}\hat{r}^{-1}_a$ is holomorphic at $\pm
  ia$ so that $R$ is holomorphic on $\P^1$ and so is constant.
\end{proof}

In particular,
\begin{equation*}
  (f_a)_b=f-\del(\hat{r}_br_a)/\del\lambda|_{\lambda=1}=
  f-\del(\hat{r}_ar_b)/\del\lambda|_{\lambda=1}=(f_{b})_a
\end{equation*}
and we have established Bianchi permutability.

\section{Isothermic surfaces}
\label{sec:isothermic-surfaces}

\subsection{Classical theory}
\label{sec:classical-theory}

First studied by Bour \cite{Bou62} in 1862, a surface
$f:\Sigma\to\R^3$
is \emph{isothermic} if it admits conformal curvature line
coordinates $x,y$ so that
\begin{align*}
  \I&:=e^{2u}(\d x^2+\d y^2)\\
  \II&=e^{2u}(\kappa_1\d x^2+\kappa_2\d y^2).
\end{align*}
A more invariant formulation is that there should exist a non-zero
holomorphic quadratic differential $q$ on $\Sigma$ such that
\begin{equation*}
  [q,\II]=0,
\end{equation*}
or, more explicitly, $[S,Q]=0$ where $Q$ is the symmetric
endomorphism with $q=I(Q\,,\,)$.  The relationship between the two
formulations is given by setting $z=x+iy$ and then $q=\d z^2$.

\begin{xmpls}
\item[]
  \begin{itemize}
  \item cones, cylinders and surfaces of revolution are isothermic:
    for the last, parametrise the profile curve in the upper half
    plane by hyperbolic arc length to get conformal curvature line
    coordinates.

    In particular, we see that isothermic surfaces have no regularity.
  \item (Stereo-images of) surfaces of constant $H$ in
    $3$-dimensional space-forms.  Here we take $q$ to be the Hopf
    differential $\II^{2,0}$.
  \item quadrics. Sadly, I know no short or conceptual argument for
    this. 
  \end{itemize}
\end{xmpls}

Isothermic surfaces have many symmetries:
\begin{enumerate}
\item Conformal invariance: if
  $\Phi:\R^3\cup\set\infty\to\R^3\cup\set\infty$ is a conformal
  diffeomorphism and $f$ is isothermic, then $\Phi\circ f$ is
  isothermic too.  This is because, while $\II$ is certainly not
  conformally invariant, its trace-free part $\II_{0}$ is and
  $[q,\II]=[q,\II_0]$.
\item In 1867, Christoffel \cite{Chr67} showed that $f$ is isothermic if and only
  if there is (locally) a \emph{dual surface} $f^c:\Sigma\to\R^3$ such
  that
  \begin{itemize}
  \item The metrics $\I$ and $\I^c$ are in the same conformal class.
  \item $f$ and $f^c$ have parallel tangent planes: $\d f(T\Sigma)=\d
    f^c(T\Sigma)$.
  \item $\det(\d f^{-1}\circ\d f^c)<0$.
  \end{itemize}
  Of course, the symmetry of the conditions means that $f^c$ is
  isothermic also and that $(f^c)^c=f$.
  \begin{xmpls}
  \item[]
    \begin{itemize}[---]
    \item When $f$ has constant mean curvature $H\neq 0$, $f^c=f+N/H$
      which has the same constant mean curvature.
    \item When $f$ is minimal, $f^c=N$, the Gauss map.

      Conversely, any conformal map $N:\Sigma\to S^2$ is isothermic
      with respect to any holomorphic quadratic differential $q$.
      Fixing such a $q$, we obtain a minimal surface
      $N^c:\Sigma\to\R^3$: this is the celebrated
      Weierstrass--Enneper formula!
    \end{itemize}
  \end{xmpls}
\item After Darboux \cite{Dar99e}, we seek a surface
  $\hf:\Sigma\to\R^3\cup\set\infty=S^3$ such that
  \begin{itemize}
  \item $f$ and $\hf$ induce the same conformal structure on $\Sigma$.
  \item $f$ and $\hf$ have the same curvature lines.
  \item For each $p\in\Sigma$, there is a $2$-sphere $S(p)\subset
    S^3$ to which both $f$ and $\hf$ are tangent at $p$.
  \end{itemize}
  In classical terminology, $f$ and $\hf$ are the enveloping surfaces
  of a \emph{conformal Ribaucour sphere congruence}.

  Here are the facts:
  \begin{enumerate}
  \item $\hf$ exists if and only if $f$ is isothermic so that, by
    symmetry, $\hf$ is isothermic also.
  \item For $a\in\R^{\times}$ and initial point $y_0\in
    S^3\setminus\set{f(p_0)}$, we can find a unique such $\hf$ with
    $\hf(p_0)=y_0$ by solving a completely integrable $5\times5$
    system of linear differential equations with a quadratic
    constraint (thus, to anticipate, finding a parallel section of a
    metric connection!).

    We write $\hf=f_a$ and call it a \emph{Darboux transformation of
      $f$ with parameter $a$}.
  \item Permutability (Bianchi \cite{Bia05}): Given
    isothermic $f$
    and two Darboux transforms $f_a$
    and $f_b$
    with $a\neq b$,
    there is a fourth isothermic surface
    $f_{ab}=(f_a)_b=(f_b)_a$
    which is algebraically determined by $f,f_a,f_b$.
    Indeed, Demoulin \cite{Dem10} shows that
    $f,f_a,f_b,f_{ab}$
    are pointwise concircular with constant cross-ratio
    $(f_b,f_a;f,f_{ab})=a/b$!
    We call a quadruple of surfaces related in this way a
    \emph{Bianchi quadrilateral}.

    One can iterate this procedure to construct a quad-graph of
    isothermic surfaces.  At each point, the corresponding quad-graph
    of points in $S^3$ with concircular elementary quadrilaterals of
    prescribed cross-ratio gives a discrete isothermic surface in the
    sense of Bobenko--Pinkall \cite{BobPin96} as we shall see in
    section~\ref{sec:discr-isoth-surf}.
    
    Moreover, if we now add a third surface $f_c$, we can apply this
    result to obtain $f_{ab},f_{ac},f_{bc}$ and then obtain an eighth
    surface $f_{abc}$ such that
    \begin{gather*}
      f_a,f_{ab},f_{ac},f_{abc}\\
      f_b,f_{ab},f_{bc},f_{abc}\\
      f_c,f_{ac},f_{bc},f_{abc}
    \end{gather*}
    are \emph{all} Bianchi quadrilaterals.
    
    This ``cube theorem'', also due to Bianchi (and also available
    for B\"acklund transformations of $K$-surfaces), can be viewed as the
    construction of a Darboux transform for discrete isothermic
    surfaces, see section~\ref{sec:gauge-theory-discr}.
  \end{enumerate}
\item Spectral deformation (Calapso 1903 \cite{Cal03},
  Bianchi 1905 \cite{Bia05a}): Given $f$
  isothermic, there is a $1$-parameter family $f_t$ of isothermic
  surfaces with $f=f_0$ inducing the same conformal structure on
  $\Sigma$ and having the same $\II_0$. The $f_t$ are called
  \emph{$T$-transforms} of $f$.
  
  Aside: We know that $\I$ and $\II$ determine a surface in $\R^3$ up
  to rigid motions.  It is therefore natural to ask if the conformal
  invariants $\<\I\>$ and $\II_0$ determine a surface in $S^3$ up to
  conformal diffeomorphism.  The $T$-transforms of an isothermic
  surface show that the answer is no but, according to
  Cartan\footnote{See \cite{BurPedPin02,BurCal10} for modern
  treatments.}, are
  the only such witnesses: if $f$ is not isothermic it is determined
  up to conformal diffeomorphism by $\<\I\>$ and $\II$.
\end{enumerate}

In this story, we recognise some familiar features: solutions in
$1$-parameter families and new solutions from commuting ODE.  We will
see how our gauge theoretic formalism applies in this situation.

\subsection{Conformal geometry (rapid introduction)}
\label{sec:conf-geom-rapid}

The conformal invariance of isothermic surfaces suggests
that we should work on the conformal compactification
$\R^3\cup\set\infty=S^3$
of $\R^3$.
For this, Darboux \cite{Dar87} offers a convenient model
which essentially linearises the situation.

Let $\R^{4,1}$ be a $5$-dimensional Minkowski space with a
metric $(\,,\,)$ of signature $++++-$ and let $\cL\subset\R^{4,1}$ be
the light-cone:
\begin{equation*}
  \cL=\set{v\in\R^{4,1}\colon (v,v)=0}.
\end{equation*}
The collection of lines through zero in $\cL$ is the projective
light-cone $\P(\cL)$ which is a smooth quadric in $\P(\R^{4,1})$
diffeomorphic to $S^3$.

$\P(\cL)$ has a conformal structure: any section
$\sigma:\P(\cL)\to\cL^{\times}$ of the projection
$\cL^{\times}\to\P(\cL)$ gives rises to a positive definite metric
\begin{equation*}
  g_{\sigma}(X,Y):=(\d_X\sigma,\d_Y\sigma)
\end{equation*}
and it is easy to see that $g_{e^u\sigma}=e^{2u}g_{\sigma}$.  Conic
sections give constant curvature metrics: more explicitly, let
$t_0\in\R^{4,1}$ have $(t_0,t_0)=-1$ and write
$\R^{4,1}=\R^4\oplus\<t_0\>$.  Then the map $x\mapsto x+t_0:S^3\to\cL$
from the unit sphere of $\R^{4}$ induces a conformal diffeomorphism
onto $\P(\cL)$.  Thus $\P(\cL)\cong S^3$ as conformal manifolds.

$k$-spheres in $S^3$ are linear objects in this picture: they are the
subsets $\P(W\cap\cL)\subset\P(\cL)$ where $W\leq\R^{4,1}$ is a
linear subspace of signature $(k+1,1)$.  For example, we obtain the
circle through three distinct points in $\P(\cL)$ by taking $W$ to be
the $(2,1)$-plane they span.
\begin{ex}
  Let $W$ be a $(3,1)$-plane.  Show that reflection across $W$
  induces the inversion of $\P(\cL)=S^3$ in the corresponding $2$-sphere.
\end{ex}
More generally, the subgroup $\rO_+(4,1)$ of the orthogonal group
that preserves the components of the light cone acts effectively by conformal
diffeomorphisms on $\P(\cL)$ and, by the exercise, has all inversions
in its image.  It follows from a theorem of Liouville that
$\rO_+(4,1)$ is the conformal diffeomorphism group of $S^3$.

In the light of all this, we henceforth treat maps
$f:\Sigma\to\P(\cL)$ and identify such maps with null line subbundles
$f\leq\underline{\R}^{4,1}$ via $f(x)=f_x$.

\subsection{Isothermic surfaces reformulated}
\label{sec:back-isoth-surf}

We give a third and final reformulation of the isothermic condition
by exploiting the structure of $S^3$ as a homogeneous space for
$G:=\rO_+(4,1)$.

Since $G$ acts transitively, we have, for $x\in S^{3}$ an isomorphism
\begin{equation*}
  T_xS^3\cong\g/\p_x,
\end{equation*}
where $\p_x$ is the infinitesimal stabiliser of $x$.  The key
algebraic ingredient in what follows is that $\p_x$ is a
\emph{parabolic} subalgebra with abelian nilradical: this means that
the polar $\p_x^{\perp}$ with respect to the Killing form is an
$\ad$-nilpotent abelian subalgebra (in fact, it is the algebra of
infinitesimal translations on the $\R^3$ obtained by stereoprojecting
away from $x$).
\begin{rem}
  We identify $\g=\fso(4,1)$ with $\Wedge^2\R^{4,1}$ via
  \begin{equation*}
    (u\wedge v)w=(u,w)v-(v,w)u
  \end{equation*}
  and then $\p_x^{\perp}=x\wedge x^{\perp}$.
\end{rem}

Now $T^{*}_xS^3\cong(\g/\p_x)^{*}$ which is isomorphic to
$\p_x^{\perp}\leq\g$ via the Killing form.  We have therefore
identified $T^{*}S^3$ with a bundle of abelian subalgebras of $\g$.

With this in hand, let $q$ be a symmetric $(2,0)$-form on $\Sigma$
and $f:\Sigma\to S^{3}$ an immersion.  Then $q+\bar{q}$ is a section
of $S^2T^{*}\Sigma$ and so may be viewed as a $1$-form with values in
$T^{*}\Sigma$.  Moreover, $\d f$ and the conformal structure of
$S^{3}$ allow us to view $T^{*}\Sigma$ as a subbundle of
$f^{-1}T^{*}S^{3}$ and so as a subbundle of $\underline{\g}$.
Chaining all this together, we see that $q+\bar{q}$ gives rise to a
$\g$-valued $1$-form $\eta$ taking values in the bundle of abelian
subalgebras $\p_f^{\perp}=f\wedge f^{\perp}$.

The crucial fact is now:
\begin{prop}
  $q$ is a holomorphic quadratic differential with $[q,\II_0]=0$ if
  and only if $\d\eta=0$.
\end{prop}
The converse is also true and we arrive at our final formulation of
the isothermic condition:
\begin{thm}
  $f$ is isothermic if and only if there is a non-zero
  $\eta\in\Omega^1(\g)$ with
  \begin{compactenum}
  \item $\d\eta=0$.
  \item $\eta$ takes values in $f\wedge f^{\perp}$.
  \end{compactenum}
\end{thm}

\subsection{Flat connections}
\label{sec:flat-connections-1}

For $f$ an isothermic surface with closed form $\eta$, we define, for
each $t\in\R$, a metric connection $\d_t=\d+t\eta$ on
$\underline{\R}^{4,1}$ and note that,
\begin{equation*}
  R^{\d_t}=R^{\d}+\d\eta+\half[\eta\wedge\eta]=0
\end{equation*}
since each summand vanishes separately: $[\eta\wedge\eta]=0$ since
$\eta$ takes values in the abelian subalgebra $f\wedge f^{\perp}$.

Thus we once again have a family of flat connections which we now exploit.

\subsection{Spectral deformation}
\label{sec:spectral-deformation-1}

Since each $\d_t$ is flat, we may locally find a trivialising gauge
$T_t:\Sigma\to\rSO(4,1)$ with $T_t\cdot\d_t=\d$.

For $s\in\R^{\times}$, we have $\d_{s+t}=\d_s+t\eta$ so that
\begin{equation*}
  \d^s_t:=T_s\cdot\d_{s+t}=\d+\Ad_{T_s}\eta
\end{equation*}
is flat for all $t$.  Set $f^s=T_sf$ and $\eta_s:=\Ad_{T_s}\eta$
which takes values in the bundle of abelian subalgebras $f_s\wedge
f_s^{\perp}$ so that $[\eta_s\wedge\eta_s]=0$.  The flatness of
$\d^s_t$ now tells us that $\d\eta_s=0$ so that $f_s$ is isothermic.

In fact, the $f_{s}$, $s\in\R^{\times}$, are the $T$-transforms of
Bianchi and Calapso.

\subsection{Parallel sections and Darboux transforms}
\label{sec:parallel-sections}

Here the analysis is much easier than for $K$-surfaces: there we
needed a slightly elaborate construction to arrive at a new surface
from parallel line-bundles.  For isothermic surfaces, the new
surfaces \emph{are} the parallel line bundles!

Indeed, for $f$ isothermic with connections $\d_t$ and
$a\in\R^{\times}$, choose a null line subbbundle
$\hf\leq\underline{\R}^{4,1}$ such that
\begin{enumerate}
  \item $\hf$ is $d_a$-parallel.
  \item $f\cap\hf=\set0$ (this condition may eventually fail far from
    the initial condition and this will introduce singularities into
    our transform).
\end{enumerate}
Then:
\begin{itemize}
\item $\hf$ is isothermic.
\item $\hf$ is a Darboux transform of $f$ with parameter $a$.
\item $\hd_t=\Gamma_f^{\hf}(1-t/a)\cdot\d_t$.
\end{itemize}

\subsection{Permutability}
\label{sec:permutability-1}

Suppose now that we have isothermic $f$ and two Darboux transforms
$f_a$ and $f_b$ with connections $\d^a_t$ and $\d^b_t$.  We seek a
fourth isothermic surface $f_{ab}$ which is a simultaneous Darboux
transform of $f_a$ and $f_b$:
\begin{equation*}
  f_{ab}=(f_a)_{b}=(f_b)_a.
\end{equation*}
Thus we need $(f_a)_b$ to be $\d^a_b$-parallel and $(f_{b})_a$ to be
$\d^b_a$-parallel.  The obvious candidates are:
\begin{align*}
  (f_a)_b&=\Gamma_f^{f_a}(1-b/a)f_b\\
  (f_b)_a&=\Gamma_f^{f_b}(1-a/b)f_a.
\end{align*}
We shall give two arguments that these coincide.

First note that, for $x,y,z\in\P(\cL)$,
\begin{equation*}
  t\mapsto\Gamma^x_y(t)z
\end{equation*}
is a rational parametrisation of the circle through $x,y,z$ by
$\R\P^1=\R\cup\set\infty$ with
\begin{align*}
  \infty&\mapsto x\\
  0&\mapsto y\\
  1&\mapsto z
\end{align*}
so that $\Gamma^x_y(t)z$ has cross-ratio $t$ with $x,y,z$:
\begin{equation*}
  (x,y;z,\Gamma^x_y(t)z)=t.
\end{equation*}
We therefore conclude that
\begin{equation}
\begin{split}
  (f_a,f;f_b,(f_a)_b)&=1-b/a\\
  (f_b,f,f_a,(f_b)_a)&=1-a/b
\end{split}\label{eq:12}
\end{equation}
whence a symmetry of the cross-ratio yields
\begin{equation*}
  (f_a,f;f_b,(f_a)_b)=(f_a,f;f_b,(f_b)_a)
\end{equation*}
so that $(f_a)_b=(f_b)_a$.

Our second argument extracts more.  We prove
\begin{equation}
  \label{eq:10}
  \Gamma_{f_a}^{(f_a)_b}(1-t/b)\Gamma^{f_a}_f(1-t/a)=
  \Gamma_{f_a}^{f_b}(\tfrac{1-t/b}{1-t/a})=
  \Gamma_{f_b}^{(f_b)_a}(1-t/a)\Gamma^{f_b}_f(1-t/b).
\end{equation}
\begin{proof}[Proof of \eqref{eq:10}]
  Let $L$ and $M$ denote the left and middle members of \eqref{eq:10}
  respectively.  It suffices to prove $L=M$ as the remaining
  equality follows by swopping the roles of $a$ and $b$.  For $L=M$,
  we note that $L$ and $M$ agree on $f_a$ and $f_a^{\perp}/f_a$ so,
  since both are orthogonal, it is enough to show that $Lf_b=f_b$ or,
  equivalently,
  \begin{equation*}
    \Gamma_f^{f_a}(1-t/a)f_b=\Gamma_{(f_a)_b}^{f_a}(1-t/b)f_b.
  \end{equation*}
  However, these are rational parametrisations of the same circle
  that agree at $\infty,0,b$ and so everywhere.
\end{proof}

We now evaluate \eqref{eq:10} on $f$ to get
\begin{equation*}
  \Gamma_{f_a}^{(f_a)_b}(1-t/b)f=
  \Gamma_{f_a}^{f_b}(\tfrac{1-t/b}{1-t/a})f=
  \Gamma_{f_b}^{(f_b)_a}(1-t/a)f
\end{equation*}
and then take $t=\infty$ to conclude
\begin{equation*}
  (f_a)_b=\Gamma_{f_a}^{f_b}(a/b)f=(f_b)_a.
\end{equation*}
As a consequence $(f_b,f_a;f,f_{ab})=a/b$ which is another version of
\eqref{eq:12}.

The same argument proves the cube theorem: introduce a third Darboux
transform $f_c$ and so three surfaces $f_{ab}, f_{bc},f_{ac}$.  This
gives rise to three new Bianchi quadrilaterals which we show share a
common surface.  To do this, evaluate \eqref{eq:10} on $f_c$ at $t=c$
to get
\begin{equation*}
  \Gamma_{f_a}^{f_{ab}}(1-c/b)f_{ac}=\Gamma_{f_a}^{f_b}(\tfrac{1-c/b}{1-c/a})f_c
  =\Gamma_{f_b}^{f_{ab}}(1-c/a)f_{bc}.
\end{equation*}
Here the left member is the simultaneous Darboux transform of
$f_{ab}$ and $f_{ac}$ and the right is the simultaneous transform of
$f_{ab}$ and $f_{bc}$.  The remaining equality follows by symmetry
and the theorem is proved.  As a bonus, we see that
$f_a,f_b,f_c,f_{abc}$ are also concircular with cross ratio
$\tfrac{1-c/b}{1-c/a}$.

\subsection{Discrete isothermic surfaces}
\label{sec:discr-isoth-surf}

View $\Z^2$ as the vertices of a combinatorial structure with edges
between adjacent vertices and quadrilateral faces.  We denote the
directed edge from $i$ to $j$ by $(j,i)$ and label the faces of
an elementary quadrilateral as follows:
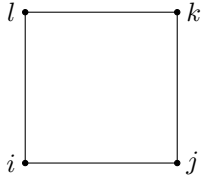
\begin{center}
  \begin{tikzpicture}[x=1.0cm,y=1.0cm]
    \draw[black] (0,0) --(0,2)--(2,2)--(2,0)--(0,0);
    \draw[fill] (0,0) node[left] {$i$} circle[radius=1pt];
    \draw[fill] (2,0) node[right] {$j$} circle[radius=1pt];
    \draw[fill] (0,2) node[left] {$l$}circle[radius=1pt];
    \draw[fill] (2,2) node[right] {$k$} circle[radius=1pt];
  \end{tikzpicture}
\end{center}
According to Bobenko--Pinkall \cite{BobPin96}, a discrete
isothermic surface is a map $f:\Z^2\to S^3=\P(\cL)$
along with a \emph{factorising function} $a$
from undirected edges to $\R^{\times}$ such that
\begin{itemize}
\item $a$ is equal on opposite edges: $a(i,j)=a(l,k)$ and
  $a(i,l)=a(j,k)$.
\item $f$ has concircular values on each elementary quadrilateral
  with cross-ratio given by
  \begin{equation*}
    (f(l),f(j);f(i),f(k))=a(i,j)/a(i,l).
  \end{equation*}
\end{itemize}
Thus the geometry is the pointwise geometry of Bianchi quadrilaterals
of smooth isothermic surfaces.

\subsection{Discrete gauge theory}
\label{sec:discr-gauge-theory}

The idea of discrete gauge theory is to replace connections by
parallel transport and vanishing curvature by trivial holonomy.  Here
are the main ingredients:
\begin{itemize}
\item A \emph{discrete vector bundle} $V$ of rank $n$ assigns a
  $n$-dimensional vector space $V_i$ to each $i\in\Z^2$.

  For example, the trivial bundle $\underline{\R}^{4,1}$ has
  $\R^{4,1}_i=\R^{4,1}$ for each $i$.
\item A \emph{section} of $V$ is a map $\sigma:\Z^2\to\bigsqcup_iV_i$
  such that $\sigma(i)\in V_i$, for all $i$.
\item A \emph{discrete connection} $\Gamma$ on $V$, assigns to each
  directed edge $(j,i)$ a linear isomorphism $\Gamma_{ji}:V_i\to V_j$
  such that
  \begin{equation*}
    \Gamma_{ij}=\Gamma_{ji}^{-1}.
  \end{equation*}
  Example: the trivial connection $1$ on $\underline{\R}^{4,1}$ has
  $1_{ji}=1$, for all edges $(j,i)$.
\item A section of $V$ is \emph{parallel} for $\Gamma$ if
  $\sigma(j)=\Gamma_{ji}\sigma(i)$, for all edges $(j,i)$.
\item A \emph{discrete gauge transformation} assigns to each $i$ a
  linear isomorphism $g(i):V_i\to V_{i}$.

  These act on connections by
  $(g\cdot\Gamma)_{ji}=g(j)\circ\Gamma_{ji}\circ g(i)^{-1}$.
\item A connection $\Gamma$ is \emph{flat} if, on every elementary
  quadrilateral we have
  \begin{equation*}
    \Gamma_{il}\Gamma_{lk}\Gamma_{kj}\Gamma_{ji}=1
  \end{equation*}
  or, equivalently,
  \begin{equation*}
    \Gamma_{kl}\Gamma_{li}=\Gamma_{kj}\Gamma_{ji}.
  \end{equation*}
  In this case, we can find a \emph{trivialising gauge}
  $T:V\to\underline{\R}^n$ such that
  \begin{equation*}
    T\cdot\Gamma=1,
  \end{equation*}
  that is, $T(i):V_i\cong\R^n$ with
  \begin{equation*}
    \Gamma_{ji}=T(j)^{-1}T(i).
  \end{equation*}
  We now have parallel sections through any point of $V$ via
  $\sigma=T^{-1}x_0$ for constant $x_0\in\R^n$.
\end{itemize}

\subsection{Gauge theory of discrete isothermic surfaces}
\label{sec:gauge-theory-discr}

Given $f:\Z^2\to S^{3}=\P(\cL)$ and a factorising function $a$ on edges,
equal on opposite edges, we define a family of connections $\Gamma^t$
on $\underline{\R}^{4,1}$ by:
\begin{equation*}
  \Gamma^t_{ji}=\Gamma_{f(i)}^{f(j)}(1-t/a(i,j)).
\end{equation*}
The arguments of section~\ref{sec:permutability-1} and especially
\eqref{eq:10} essentially establish the following result:
\begin{thm}
  $f$ is discrete isothermic with factorising function $a$ if and only if
  $\Gamma^t$ is flat for all $t\in\R$.
\end{thm}

We may now apply all our previous gauge theoretic arguments in this
new setting!  We give just one example: a Darboux transform of a
discrete isothermic surface should be given by a parallel null line
subbundle.  To verify this, fix $\hat{a}\in\R^{\times}$ and let
$\hf\leq\underline{\R}^{4,1}$ be a null line subbundle such that
\begin{enumerate}
\item $\hf$ is $\Gamma^{\hat{a}}$-parallel.
\item $f(i)\cap\hf=\set0$, for all $i$.
\end{enumerate}
Spelling out the parallel condition gives
\begin{equation*}
  \Gamma_{f(i)}^{f(j)}(1-\hat{a}/a(i,j))\hf(i)=\hf(j)
\end{equation*}
so that $f(i),f(j),\hf(i),\hf(j)$ are concircular with fixed cross
ratio $a(i,j)/\hat{a}$.  Relabelling \eqref{eq:10} to describe this
quadrilateral gives
\begin{equation*}
  \Gamma_{f(j)}^{\hf(j)}(1-t/\hat{a})\Gamma_{f(i)}^{f(j)}(1-t/a(i,j))=
  \Gamma_{\hf(i)}^{\hf(j)}(1-t/a(i,j))\Gamma_{f(i)}^{f(i)}(1-t/\hat{a}).
\end{equation*}
Otherwise said:
\begin{equation*}
  \Gamma_f^{\hf}(1-t/\hat{a})\cdot \Gamma^{t}=\hat{\Gamma}^t.
\end{equation*}
Now $\hat{\Gamma}^t$ is flat for all $t$ being a gauge of flat
$\Gamma^t$ so that $\hf$ is indeed isothermic with the same
factorising function as $f$.

We remark that we can iterate this construction and so build up a map
$F:\Z^3\to S^3$ whose restrictions to level sets
$\set{(n_1,n_2,n_3)\in\Z^3\colon n_3=m}$ are our iterated Darboux
transforms.  Moreover, the two other families of level sets obtained
by holding $n_1$ or $n_2$ fixed also consist (as we have just seen)
of concircular quadrilaterals with factorising cross-ratios and so
are isothermic also.  We therefore have a triple system of discrete
isothermic surfaces!

\begin{ex}
  Find a spectral deformation for discrete isothermic surfaces.
\end{ex}


\end{document}